\documentclass[a4paper,11pt]{amsart}
\usepackage[T1]{fontenc}
\usepackage[english]{babel}
\usepackage[cp1252]{inputenc}
\usepackage{amsthm}
\usepackage{amsmath}
\usepackage{amsfonts}
\usepackage{amssymb}
\usepackage{hyperref}
\usepackage{color}
\usepackage{caption}
\usepackage{indentfirst}
\usepackage{amssymb}
\usepackage{eufrak}
\usepackage{mathrsfs}
\usepackage{xypic}
\usepackage[pdftex]{graphicx}
\usepackage{tikz}
\usepackage{booktabs}

\theoremstyle{plain}

\newcommand{\R}{\mathbb R}
\newcommand{\C}{\mathbb C}

\newcommand{\Z}{\mathbb Z}

\newcommand{\Log}{\mathrm{Log}\,}

\newcommand{\PP}{\mathbb P}

\newcommand{\Crit}{\operatorname{Crit}}

\newcommand{\CA}{\mathcal C\mathcal A}

\newcommand{\Sing}{\operatorname{Sing}}
\newcommand{\Reg}{\operatorname{Reg}}

\newcommand{\rank}{\operatorname{rank}}

\newcommand\restr[2]{{
  \left.\kern-\nulldelimiterspace 
  #1 
  \right|_{#2} 
  }}

\newtheorem{criterion}{Criterion}

\newtheorem{remark}{Remark}[section]
\newtheorem*{mtheorem*}{Main Theorem}
\newtheorem{theorem}{Theorem}[section]

\newtheorem{proposition}{Proposition}[section]

\newtheorem{lemma}{Lemma}[section]

\begin{document}
\title{On the Analyticity of Amoeba Contours} 
\author{Mounir Nisse}

\address{Mounir Nisse\\
Department of Mathematics, Xiamen University Malaysia, Jalan Sunsuria, Bandar Sunsuria, 43900, Sepang, Selangor, Malaysia.
}
\email{mounir.nisse@gmail.com, mounir.nisse@xmu.edu.my}

\thanks{}  



\maketitle

\begin{abstract}
The contour of an amoeba records the critical values of the logarithmic map and forms a fundamental interface between complex algebraic geometry and real-analytic geometry. Building on the known semianalyticity of amoeba contours, we prove the stronger statement that the componentwise exponential image of the contour of any algebraic hypersurface in $(\mathbb C^*)^n$ is semialgebraic. For a smooth hypersurface, we establish a natural criterion guaranteeing that the contour is a closed real-analytic hypersurface: the logarithmic Gauss map is transverse to $\mathbb RP^{n-1}$ and the logarithmic map has maximal rank along its critical locus. We then show that neither hypothesis is necessary, since singular or ramified critical parametrizations may still have complete real-analytic images. Finally, we derive intrinsic restrictions imposed by analyticity. In particular, every irreducible contour component of dimension $n-1$ is generated, over a dense open subset of its regular locus, by a real-analytic critical stratum on which the logarithmic map has maximal rank. These results distinguish the analytic geometry of the contour from the singularities of its critical parametrization and answer the question of Lang, Shapiro, and Shustin concerning effective sufficient conditions for analyticity.
\end{abstract}


\section*{Introduction and preliminaries}

Amoebas associate real geometric objects with complex algebraic varieties through the logarithmic map. If $V\subset(\C^*)^n$ is algebraic, its amoeba is $\mathcal A_V=\Log(V)\subset\R^n$, where $\Log(z_1,\ldots,z_n)=(\log|z_1|,\ldots,\log|z_n|)$. The critical values of $\Log$ form the contour $\CA_V$, a fundamental subset governing the geometry of the amoeba and the singular behavior of its logarithmic projection; see \cite{Mikhalkin00,Mikhalkin04,Mikhalkin04Pants,Lang19,LangShapiroShustin21}.

The main purpose of this note is to determine the precise analytic nature of amoeba contours. Our first result proves that if $V=\{f=0\}\subset(\C^*)^n$ is a reduced algebraic hypersurface, then $\CA_V=\Log(\Crit(\Log|_{V_{\mathrm{sm}}}))$ is always semianalytic. We establish the stronger statement that $\operatorname{Exp}(\CA_V)\subset(0,\infty)^n$ is semialgebraic. The proof replaces the angular variables by algebraic coordinates on products of circles, describes the logarithmic critical locus by polynomial equations, inequalities, and Jacobian minors, and applies the Tarski--Seidenberg theorem. This shows that semianalyticity is the optimal unconditional regularity supplied by algebraicity.

Our second result gives a natural sufficient criterion for the contour to be analytic. For a smooth hypersurface $H=\{f=0\}$, let $\gamma_H:H\to\C\PP^{n-1}$ be the logarithmic Gauss map given by $\gamma_H(z)=[z_1f_{z_1}(z):\cdots:z_nf_{z_n}(z)]$. Since $C_H=\Crit(\Log|_H)=\gamma_H^{-1}(\R\PP^{n-1})$, transversality of $\gamma_H$ to $\R\PP^{n-1}$ makes $C_H$ a smooth real-analytic manifold of dimension $n-1$. We prove that if $\Log|_{C_H}$ has rank $n-1$ everywhere, then $\CA_H$ is a closed real-analytic hypersurface of $\R^n$. The contour may have self-intersections or several local branches, but properness of the logarithmic map ensures that only finitely many branches occur near each point and that their union is defined by one real-analytic equation.

We then show that neither transversality of the logarithmic Gauss map nor immersion of the entire critical locus is necessary for analyticity of the contour. This distinction is essential: a singular or ramified critical locus may still have a complete analytic image. We therefore derive intrinsic necessary conditions on the contour itself. A real-analytic contour must be locally closed, locally finite in its branch structure, free of incomplete one-sided image germs, and have real-algebraic tangent cones. Most importantly, if $Y$ is an irreducible local component of $\CA_H$ of dimension $n-1$, then there exists a dense open subset $Y^\circ\subset\Reg(Y)$ such that every $x\in Y^\circ$ has a preimage $p\in C_H\cap\Log^{-1}(x)$ lying on a real-analytic stratum $S\subset C_H$ for which $\operatorname{rank}d(\Log|_S)_p=n-1$. Thus maximal rank is necessary generically on at least one contributing stratum, although it may fail at exceptional ramification points.

We recall only the notions needed for these results. A subset of a real-analytic manifold is analytic if it is locally the common zero set of finitely many real-analytic functions, and it is semianalytic if it is locally a finite union of sets defined by analytic equalities and inequalities. Semialgebraic sets are defined similarly using polynomials and are preserved by polynomial projections \cite{BCR98}. Proper real-analytic images of semianalytic sets are generally subanalytic rather than semianalytic \cite{Hironaka73,BM88}; the stronger conclusion obtained here comes from the semialgebraic description before applying the logarithm.

The proofs also use two structural facts. The restriction of $\Log$ to a closed algebraic hypersurface is proper because inverse images of compact subsets of $\R^n$ lie in compact products of annuli. Moreover, the critical locus admits a locally finite real-analytic Whitney stratification compatible with the rank loci of $\Log$ \cite{Hardt75,Hironaka73,Whitney65}. The images of strata of rank at most $n-2$ have dimension at most $n-2$ and therefore cannot contain an open subset of an analytic contour component of dimension $n-1$. Together, algebraic elimination, properness, and rank-adapted stratification provide a concise description of the passage from semianalyticity to analyticity and isolate the genuinely necessary geometric conditions.


\section{The contour of an amoeba is a semianalytic subset}

We give a complete proof for algebraic hypersurfaces in the complex algebraic torus. The statement is understood with the standard definition of the contour as the set of critical values of the logarithmic map restricted to the smooth locus. The argument also explains why it is important that the hypersurface is algebraic: after replacing logarithmic coordinates by positive modulus coordinates, the critical locus and its image are governed by semialgebraic geometry.

 A subset $A$ of a real-analytic manifold $M$ is a real-analytic subset if, for every $a\in M$, there are an open neighborhood $U$ of $a$ and finitely many real-analytic functions $g_1,\ldots,g_s$ on $U$ such that $A\cap U=\{x\in U:g_1(x)=\cdots=g_s(x)=0\}$. In particular, an analytic subset is locally closed in the ambient space and is locally defined by equalities only. This is substantially stronger than semianalyticity, which permits finite Boolean combinations of analytic equalities and inequalities.

\begin{theorem}
Let $n\geq 2$, let $f\in\C[z_1^{\pm1},\ldots,z_n^{\pm1}]$ be a nonzero Laurent polynomial, and let $V=\{z\in(\C^\ast)^n:f(z)=0\}$. Denote by $V_{\mathrm{sm}}$ the smooth locus of the reduced hypersurface associated with $V$. Let $\Log:(\C^\ast)^n\to\R^n$ be defined by $\Log(z_1,\ldots,z_n)=(\log|z_1|,\ldots,\log|z_n|)$. Then the contour
$
\mathcal C\mathcal A_V=\Log\bigl(\Crit(\Log|_{V_{\mathrm{sm}}})\bigr)
$
is a semianalytic subset of $\R^n$. In fact, its image under the componentwise exponential map $\operatorname{Exp}:\R^n\to(0,\infty)^n$, $\operatorname{Exp}(x_1,\ldots,x_n)=(e^{x_1},\ldots,e^{x_n})$, is semialgebraic.
\end{theorem}

\begin{proof}
Recall that a subset $S$ of a real analytic manifold $M$ is semianalytic if every point $p\in M$ has a neighborhood $U$ such that $S\cap U$ is a finite union of sets defined by finitely many real analytic equalities and strict inequalities. Thus, locally, a semianalytic set has the form of a finite union of sets
$
\{x\in U:g_1(x)=\cdots=g_a(x)=0,\ h_1(x)>0,\ldots,h_b(x)>0\},
$
where the functions $g_i$ and $h_j$ are real analytic on $U$. Every semialgebraic subset of a Euclidean space is semianalytic because polynomial functions are real analytic.

We begin by replacing the polar angles by algebraic variables. For $r=(r_1,\ldots,r_n)\in(0,\infty)^n$ and $(u,v)=(u_1,v_1,\ldots,u_n,v_n)\in\R^{2n}$ satisfying $u_j^2+v_j^2=1$, put $z_j=r_j(u_j+iv_j)$. Consider the real algebraic set
$
T=\{(r,u,v)\in(0,\infty)^n\times\R^{2n}:u_j^2+v_j^2=1\text{ for }1\leq j\leq n\}.
$
The map $\Phi:T\to(\C^\ast)^n$ defined by $\Phi(r,u,v)=(r_1(u_1+iv_1),\ldots,r_n(u_n+iv_n))$ is a real analytic diffeomorphism. Its inverse is given by $r_j=|z_j|$, $u_j=\operatorname{Re}(z_j)/|z_j|$, and $v_j=\operatorname{Im}(z_j)/|z_j|$.

Since $f$ is Laurent, there is a vector $N=(N_1,\ldots,N_n)\in\mathbb Z_{\geq0}^n$ such that $P(z)=z_1^{N_1}\cdots z_n^{N_n}f(z)$ is an ordinary polynomial. The monomial used here never vanishes on $(\C^\ast)^n$, so $\{f=0\}$ and $\{P=0\}$ define the same subset of the torus. Write
$
P\bigl(r_1(u_1+iv_1),\ldots,r_n(u_n+iv_n)\bigr)=F_1(r,u,v)+iF_2(r,u,v).
$
The functions $F_1$ and $F_2$ are real polynomials in the $3n$ real variables $r_1,\ldots,r_n,u_1,v_1,\ldots,u_n,v_n$. Consequently, the inverse image $\Phi^{-1}(V)$ is the semialgebraic subset of $T$ defined by $F_1=F_2=0$.

The smoothness condition is also semialgebraic. A point $z\in V$ is smooth on the reduced hypersurface if and only if the differential of a local reduced defining equation does not vanish there. Replacing $P$ by its square-free part does not change the underlying reduced hypersurface, so we may and shall assume that $P$ is reduced. On the torus, a point of $V$ is smooth precisely when the complex gradient $(P_{z_1},\ldots,P_{z_n})$ is not the zero vector. After the substitution $z_j=r_j(u_j+iv_j)$, this condition can be written as
$
\sum_{j=1}^n\left(\operatorname{Re}P_{z_j}\right)^2+
\sum_{j=1}^n\left(\operatorname{Im}P_{z_j}\right)^2>0.
$
The expression on the left becomes a real polynomial in $(r,u,v)$. It follows that the set $X=\Phi^{-1}(V_{\mathrm{sm}})$ is semialgebraic.

We now prove that the critical locus is semialgebraic without choosing local coordinates on $V_{\mathrm{sm}}$. Define $q_j(r,u,v)=u_j^2+v_j^2-1$ and let
$
G=(F_1,F_2,q_1,\ldots,q_n):\R^{3n}\longrightarrow\R^{n+2}.
$
Near every point of $X$, the equations $G=0$ define $X$ as a smooth real submanifold of dimension $2n-2$. Indeed, the equations $q_j=0$ define the smooth manifold $T$, the map $\Phi:T\to(\C^\ast)^n$ is a diffeomorphism, and $F_1=F_2=0$ cut out the smooth complex hypersurface $V_{\mathrm{sm}}$, which has real codimension two. Therefore $dG$ has constant rank $n+2$ at every point of $X$, and $T_pX=\ker dG_p$.

Let $\rho:X\to(0,\infty)^n$ be the restriction of the coordinate projection $\rho(r,u,v)=r$. Since $\Log\circ\Phi=\log\circ\rho$, where $\log:(0,\infty)^n\to\R^n$ is the componentwise logarithm, and since $d\log_r=\operatorname{diag}(1/r_1,\ldots,1/r_n)$ is invertible at every $r\in(0,\infty)^n$, a point $p\in X$ is critical for $\Log\circ\Phi$ if and only if it is critical for $\rho$. Thus
$
\Phi^{-1}\bigl(\Crit(\Log|_{V_{\mathrm{sm}}})\bigr)=\Crit(\rho|_X).
$

We express this last criticality condition by the vanishing of minors. Let $D(G,\rho)_p$ denote the Jacobian matrix of the map $(G,\rho):\R^{3n}\to\R^{2n+2}$ at $p$. Because $\rank dG_p=n+2$, the rank of $d\rho_p|_{T_pX}$ is
$
\rank d\rho_p|_{\ker dG_p}=\rank D(G,\rho)_p-\rank dG_p
=\rank D(G,\rho)_p-(n+2).
$
For completeness, this linear-algebra identity follows by considering the linear map $A=dG_p:E\to\R^{n+2}$ and $B=d\rho_p:E\to\R^n$, where $E=\R^{3n}$. The quotient of the image of $(A,B)$ by the subspace $\operatorname{im}A\times\{0\}$ is naturally isomorphic to $B(\ker A)$, and hence $\rank(A,B)=\rank A+\rank(B|_{\ker A})$.

Since the target of $\rho$ has dimension $n$, the point $p$ is critical for $\rho|_X$ exactly when $\rank d\rho_p|_{T_pX}<n$. By the preceding identity, this is equivalent to $\rank D(G,\rho)_p<2n+2$. The entries of $D(G,\rho)$ are polynomials. Therefore this rank condition is equivalent to the simultaneous vanishing of all $(2n+2)\times(2n+2)$ minors of $D(G,\rho)$. If $2n+2>3n$, which does not occur for $n\geq2$, the rank condition would be automatic; for $n\geq2$ the displayed minors give the required polynomial equations. It follows that
$
\Sigma=\Crit(\rho|_X)
$
is a semialgebraic subset of $\R^{3n}$, because it is defined by the polynomial equations $F_1=F_2=0$ and $q_j=0$, the inequalities $r_j>0$, the polynomial inequality expressing smoothness, and the vanishing of finitely many Jacobian minors.

Let $\pi_r:\R^{3n}\to\R^n$ be the projection onto the $r$-coordinates and put $B=\pi_r(\Sigma)$. The Tarski--Seidenberg theorem asserts that the image of a semialgebraic set under a polynomial map, and in particular under a coordinate projection, is semialgebraic. Hence $B$ is a semialgebraic subset of $\R^n$. Since every point of $\Sigma$ satisfies $r_j>0$, we have $B\subset(0,\infty)^n$.

By construction, $B$ is precisely the coordinatewise modulus image of the logarithmic critical locus:
$
B=\{(|z_1|,\ldots,|z_n|):z\in\Crit(\Log|_{V_{\mathrm{sm}}})\}.
$
Consequently,
$
\mathcal C\mathcal A_V=\log(B),
$
where $\log:(0,\infty)^n\to\R^n$ is the componentwise logarithm. Equivalently, $\operatorname{Exp}(\mathcal C\mathcal A_V)=B$, proving the stronger assertion that the exponential image of the contour is semialgebraic.

It remains to verify carefully that $\log(B)$ is semianalytic. Fix $x^0\in\R^n$ and put $r^0=\operatorname{Exp}(x^0)$. Because $B$ is semialgebraic, there is a neighborhood $U$ of $r^0$ in $(0,\infty)^n$ on which $B\cap U$ is a finite union of sets defined by polynomial equalities and strict polynomial inequalities. Let $W=\log(U)$, which is a neighborhood of $x^0$. If a piece of $B\cap U$ is defined by $p_i(r)=0$ and $q_j(r)>0$, then its image under $\log$ is defined on $W$ by
$
p_i(e^{x_1},\ldots,e^{x_n})=0
\quad\text{and}\quad
q_j(e^{x_1},\ldots,e^{x_n})>0.
$
Each function $x\mapsto p_i(e^{x_1},\ldots,e^{x_n})$ and $x\mapsto q_j(e^{x_1},\ldots,e^{x_n})$ is real analytic on $\R^n$. Therefore $\log(B)\cap W$ is a finite union of basic semianalytic sets. Since $x^0$ was arbitrary, $\log(B)$ is semianalytic in $\R^n$. This proves that $\mathcal C\mathcal A_V$ is semianalytic.
\end{proof}

\begin{remark}
A general proper real analytic image of a semianalytic set is known to be subanalytic, but such an image need not be semianalytic. In the algebraic amoeba setting, semianalyticity follows from the stronger fact that, before applying the logarithm, the modulus image of the critical locus is semialgebraic. The Tarski--Seidenberg theorem is therefore the decisive ingredient of the conclusion from subanalytic to semianalytic.
\end{remark}



\section{Regular and Analytic Amoeba Contours}
There is no purely combinatorial classification, in terms only of the Newton polytope or the degree, of all algebraic varieties whose amoeba contours are analytic. Analyticity is controlled by the local singularity theory of the logarithmic projection on its critical locus. The correct general statement is that the contour is semianalytic, whereas analyticity follows when every local critical branch has a complete real-analytic image and only finitely many such images occur above each contour point. A particularly useful geometric hypothesis is that the critical locus be smooth and that the logarithmic map restrict to an immersion on it. We formulate and prove this criterion in detail.

Let $H=\{f=0\}\subset(\C^\ast)^n$ be a smooth algebraic hypersurface, where $n\geq2$. The logarithmic map is $\Log(z_1,\ldots,z_n)=(\log|z_1|,\ldots,\log|z_n|)$. Its critical locus and contour are
$
C_H=\Crit(\Log|_H)
$
and
$
\mathcal C\mathcal A_H=\Log(C_H).
$
The logarithmic Gauss map is
$
\gamma_H:H\longrightarrow\C\PP^{n-1},\qquad
\gamma_H(z)=[z_1f_{z_1}(z):\cdots:z_nf_{z_n}(z)].
$
The standard logarithmic-Gauss characterization gives
$
C_H=\gamma_H^{-1}(\R\PP^{n-1}).
$
This description is the starting point for the analytic criterion.
   

%

\begin{theorem}\label{thm:1}      
Let $H\subset(\C^\ast)^n$ be a smooth algebraic hypersurface. Assume that $\gamma_H$ is transverse to $\R\PP^{n-1}$. Assume moreover that
$
\rank d(\Log|_{C_H})_p=n-1
$
for every $p\in C_H$. Then $\mathcal C\mathcal A_H$ is a closed real-analytic hypersurface of $\R^n$, possibly with finitely many local branches and self-intersections at a given point.
\end{theorem}

Before proving the theorem, we explain the meaning and role of its hypotheses. The equality $C_H=\gamma_H^{-1}(\R\PP^{n-1})$ identifies logarithmic criticality with a real condition on the logarithmic normal direction. Transversality of $\gamma_H$ to $\R\PP^{n-1}$ means that, at each $p\in C_H$,
$
d\gamma_H(T_pH)+T_{\gamma_H(p)}\R\PP^{n-1}
=T_{\gamma_H(p)}\C\PP^{n-1}.
$
This prevents the inverse image from having an unexpected singularity or an excessive dimension. It implies that $C_H$ is a smooth real-analytic manifold of dimension $n-1$. The second hypothesis says that $\Log$ does not ramify when restricted to this $(n-1)$-dimensional manifold. Since the rank equals the dimension of the source, $\Log|_{C_H}$ is an immersion. Its local images are therefore embedded real-analytic hypersurfaces. Properness of the logarithmic map then guarantees that only finitely many such local images occur above a given contour point and that no distant branches enter after the target neighborhood is made smaller.

\begin{lemma}[Mikhalkin \cite{Mikhalkin04, Mikhalkin04Pants}]
For every $p\in H$,
$
p\in C_H
$
if and only if
$
\gamma_H(p)\in\R\PP^{n-1}.
$
Consequently,
$
C_H=\gamma_H^{-1}(\R\PP^{n-1}).
$
\end{lemma}
\begin{proof}
Fix $p=(p_1,\ldots,p_n)\in H$. Choose a simply connected neighborhood of $p$ in $(\C^\ast)^n$ on which holomorphic logarithms are defined, and introduce local logarithmic coordinates $\zeta_j=\log z_j$. In these coordinates, the differential of the ambient logarithmic map is
$
d\Log_p(v_1,\ldots,v_n)=(\operatorname{Re}v_1,\ldots,\operatorname{Re}v_n).
$
Indeed, if $z_j(s)=p_je^{sv_j}$, then the derivative of $\log|z_j(s)|$ at $s=0$ is $\operatorname{Re}v_j$.

Put
$
a_j=p_jf_{z_j}(p)
$
and
$
a=(a_1,\ldots,a_n)\in\C^n.
$
In logarithmic coordinates, the complex tangent space of $H$ at $p$ is the complex hyperplane
$
T_pH=\{v\in\C^n:a_1v_1+\cdots+a_nv_n=0\}.
$
This follows from the chain rule, since $\partial f/\partial\zeta_j=z_jf_{z_j}$. The vector $a$ is nonzero because $H$ is smooth.

The point $p$ is critical for $\Log|_H$ precisely when the real-linear map $d(\Log|_H)_p:T_pH\to\R^n$ is not surjective. This is equivalent to the existence of a nonzero real covector $\xi=(\xi_1,\ldots,\xi_n)\in\R^n$ that annihilates its image. Hence $p$ is critical precisely when there exists $\xi\neq0$ such that
$
\xi_1\operatorname{Re}v_1+\cdots+\xi_n\operatorname{Re}v_n=0
$
for every $v\in T_pH$.
Since the numbers $\xi_j$ are real, the left-hand side is $\operatorname{Re}(\xi_1v_1+\cdots+\xi_nv_n)$. Moreover, $T_pH$ is a complex vector space. Thus, whenever $v\in T_pH$, one also has $iv\in T_pH$. Applying the annihilation condition first to $v$ and then to $iv$ gives
$
\operatorname{Re}\left(\sum_{j=1}^n\xi_jv_j\right)=0
$
and
$
\operatorname{Re}\left(i\sum_{j=1}^n\xi_jv_j\right)
=-\operatorname{Im}\left(\sum_{j=1}^n\xi_jv_j\right)=0.
$
Therefore
$
\sum_{j=1}^n\xi_jv_j=0
$
for every $v\in T_pH$. The complex covector $\xi$ and the nonzero complex covector $a$ have the same kernel, namely the complex hyperplane $T_pH$. Hence they are proportional over $\C$: there is a $\lambda\in\C^\ast$ such that $\xi=\lambda a$. Since $\xi$ has real coordinates, the projective class $[a_1:\cdots:a_n]$ belongs to $\R\PP^{n-1}$. But this projective class is exactly $\gamma_H(p)$.

Conversely, assume that $\gamma_H(p)\in\R\PP^{n-1}$. Then there are a nonzero vector $\xi\in\R^n$ and a scalar $\lambda\in\C^\ast$ such that $\xi=\lambda a$. Every $v\in T_pH$ satisfies $a\cdot v=0$, and therefore $\xi\cdot v=0$. Taking real parts gives $\xi\cdot d\Log_p(v)=0$. Hence the image of $d(\Log|_H)_p$ is contained in the proper real hyperplane $\xi^\perp\subset\R^n$, so the differential is not surjective. Thus $p\in C_H$.
\end{proof}

\begin{lemma}%
If $\gamma_H$ is transverse to $\R\PP^{n-1}$, then $C_H$ is a closed real-analytic submanifold of $H$ of real dimension $n-1$.
\end{lemma}

\begin{proof}
The complex manifold $H$ has complex dimension $n-1$ and real dimension $2n-2$. The complex projective space $\C\PP^{n-1}$ also has real dimension $2n-2$, while $\R\PP^{n-1}$ has real dimension $n-1$. Therefore the real codimension of $\R\PP^{n-1}$ in $\C\PP^{n-1}$ is
$
(2n-2)-(n-1)=n-1.
$

The logarithmic Gauss map is real analytic. By the real-analytic transverse-preimage theorem, the inverse image of a real-analytic submanifold under a transverse real-analytic map is a real-analytic submanifold whose codimension equals the codimension of the target submanifold. Using the preceding lemma, we obtain
$
C_H=\gamma_H^{-1}(\R\PP^{n-1}),
$
and hence
$
\dim_\R C_H
=\dim_\R H-\operatorname{codim}_\R(\R\PP^{n-1},\C\PP^{n-1})
=(2n-2)-(n-1)=n-1.
$
The subset $\R\PP^{n-1}$ is closed in $\C\PP^{n-1}$. Since $\gamma_H$ is continuous, its inverse image $C_H$ is closed in $H$.
\end{proof}

\begin{lemma} %
The restriction $\Log|_H:H\to\R^n$ is proper. Consequently, its restriction
$
F=\Log|_{C_H}:C_H\longrightarrow\R^n
$
is proper.
\end{lemma}

\begin{proof}
Let $K\subset\R^n$ be compact. There is an $M>0$ such that $|x_j|\leq M$ for all $x=(x_1,\ldots,x_n)\in K$ and all $j$. If $z\in\Log^{-1}(K)$, then
$
e^{-M}\leq|z_j|\leq e^M
$
for every $j$. Thus $\Log^{-1}(K)$ is contained in the compact product of closed annuli
$
A_M=\{z\in\C^n:e^{-M}\leq|z_j|\leq e^M\text{ for every }j\}.
$
The set $H\cap\Log^{-1}(K)$ is closed in $A_M$. Indeed, $H$ is the zero set of a Laurent polynomial in the torus, and all coordinates on $A_M$ are bounded away from zero, so $H\cap A_M$ is closed in $A_M$. The set $\Log^{-1}(K)$ is also closed by continuity. Therefore
$
(\Log|_H)^{-1}(K)=H\cap\Log^{-1}(K)
$
is compact. This proves that $\Log|_H$ is proper.
Since $C_H$ is closed in $H$, the set
$
F^{-1}(K)=C_H\cap(\Log|_H)^{-1}(K)
$
is closed in a compact set and hence compact. Thus $F$ is proper.
\end{proof}

\begin{proof}[Proof of  Theorem~\ref{thm:1}]
By the logarithmic-Gauss characterization and the transversality hypothesis, $C_H$ is a closed real-analytic manifold of dimension $n-1$. Define
$
F=\Log|_{C_H}:C_H\to\R^n.
$
The hypothesis $\rank dF_p=n-1$ for every $p\in C_H$ says that the differential has rank equal to the dimension of its source. Therefore $F$ is a real-analytic immersion.

Fix a contour point $x\in\mathcal C\mathcal A_H$. By definition, $F^{-1}(x)$ is nonempty. Properness implies that this fiber is compact. Since $F$ is an immersion, the real-analytic constant-rank theorem implies that $F$ is locally an embedding near every point of $C_H$. In particular, every point of $F^{-1}(x)$ is isolated in that fiber. A compact discrete subset of a manifold is finite. To see this directly, an infinite subset of a compact metric space has an accumulation point, while a discrete subset has no accumulation point. Hence
$
F^{-1}(x)=\{p_1,\ldots,p_r\}
$
for some finite integer $r\geq1$.

For each $p_j$, the constant-rank theorem provides a neighborhood $U_j$ of $p_j$ in $C_H$ and real-analytic coordinates $(t_1,\ldots,t_{n-1})$ on $U_j$ and $(y_1,\ldots,y_n)$ near $x$ such that
$
F(t_1,\ldots,t_{n-1})
=(t_1,\ldots,t_{n-1},0).
$
After shrinking the $U_j$, they may be assumed pairwise disjoint, and every restriction $F|_{U_j}$ is a real-analytic embedding. Its image
$
S_j=F(U_j)
$
is therefore an embedded real-analytic hypersurface germ through $x$.

We now prove that no additional part of $C_H$ can map arbitrarily close to $x$. There is an open neighborhood $W$ of $x$ such that
$
F^{-1}(W)\subset U_1\cup\cdots\cup U_r.
$
Suppose this were false. Choose a decreasing sequence of relatively compact neighborhoods $W_\nu$ of $x$ with intersection $\{x\}$. For every $\nu$, there would be a point
$
q_\nu\in F^{-1}(W_\nu)\setminus(U_1\cup\cdots\cup U_r).
$
Then $F(q_\nu)\to x$. The set consisting of $x$ and the points $F(q_\nu)$ has compact closure $K$. Properness makes $F^{-1}(K)$ compact, so a subsequence of $q_\nu$ converges to a point $q\in C_H$. Continuity gives $F(q)=x$, and therefore $q=p_j$ for some $j$. Since $U_j$ is a neighborhood of $p_j$, the convergent subsequence eventually lies in $U_j$, contradicting its construction. The asserted neighborhood $W$ therefore exists.

After shrinking $W$ and the $U_j$ once more, we have
$
\mathcal C\mathcal A_H\cap W
=F(C_H)\cap W
=\bigcup_{j=1}^r(F(U_j)\cap W)
=\bigcup_{j=1}^r(S_j\cap W).
$
Each $S_j\cap W$ is an embedded real-analytic hypersurface. Hence there exists a real-analytic function $g_j:W\to\R$ such that
$
S_j\cap W=\{y\in W:g_j(y)=0\}
$
and $dg_j$ does not vanish along $S_j\cap W$. The union of these branches is the zero set of the product:
$
\bigcup_{j=1}^r(S_j\cap W)
=\{y\in W:g_1(y)\cdots g_r(y)=0\}.
$
The product $g_1\cdots g_r$ is real analytic. Thus $\mathcal C\mathcal A_H$ is a real-analytic hypersurface near $x$.

This conclusion remains valid when two or more image branches coincide or meet each other. Coincident branches only repeat a factor and may be removed. Distinct branches may intersect transversally or tangentially, but their union is still defined by the product of their local equations. Such an intersection is generally a singular point of the contour as an analytic hypersurface, although each individual branch is smooth. The number of local branches is finite because the fiber $F^{-1}(x)$ is finite.
Since $x$ was arbitrary, $\mathcal C\mathcal A_H$ is a real-analytic hypersurface of $\R^n$. It remains to prove closedness. Let $x_\nu\in\mathcal C\mathcal A_H$ converge to $x\in\R^n$. Choose $p_\nu\in C_H$ with $F(p_\nu)=x_\nu$. The set consisting of $x$ and the sequence $x_\nu$ has compact closure $K$. Properness implies that $F^{-1}(K)$ is compact, so a subsequence $p_{\nu_k}$ converges to some $p\in C_H$. Continuity gives $F(p)=x$. Hence $x\in F(C_H)=\mathcal C\mathcal A_H$. Therefore the contour is closed.
\end{proof}


\section{Are the transversality and immersion hypotheses necessary?}

Let $H=\{f=0\}\subset(\C^\ast)^n$ be a smooth algebraic hypersurface, let
$
\gamma_H(z)=[z_1f_{z_1}(z):\cdots:z_nf_{z_n}(z)]
$
be its logarithmic Gauss map, and put
$
C_H=\Crit(\Log|_H)=\gamma_H^{-1}(\R\PP^{n-1}).
$
The immersed-critical-locus theorem assumes that $\gamma_H$ is transverse to $\R\PP^{n-1}$ and that $\rank d(\Log|_{C_H})_p=n-1$ at every $p\in C_H$. These hypotheses imply that $C_H$ is a smooth $(n-1)$-dimensional real-analytic manifold and that its logarithmic image is locally a finite union of immersed analytic hypersurfaces. They are strong, convenient sufficient hypotheses. They are not necessary for the contour to be a closed real-analytic hypersurface.

\begin{theorem}
Neither the condition that $\gamma_H$ be transverse to $\R\PP^{n-1}$ nor the condition\\
$
\rank d(\Log|_{C_H})_p=n-1
$
at every $p\in C_H$ is necessary for $\mathcal C\mathcal A_H$ to be a closed real-analytic hypersurface. The first condition can fail because the logarithmic Gauss map is excessively degenerate while its image contour remains an affine hyperplane. The second can fail because a ramified critical branch may have a complete real-analytic image, for example an ordinary cusp. What is necessary is a condition on the image germs, not maximal rank of every parametrizing branch.
\end{theorem}

\begin{proof}
We first disprove necessity of transversality by an explicit family. Let $\alpha=(\alpha_1,\ldots,\alpha_n)\in\Z^n$ be primitive, let $c_0,c_1\in\C^\ast$, and consider the binomial hypersurface
$
H_{\alpha,c}=\{z\in(\C^\ast)^n:c_0+c_1z^\alpha=0\},
$
where $z^\alpha=z_1^{\alpha_1}\cdots z_n^{\alpha_n}$. Since $\alpha$ is primitive, the character $\chi_\alpha:(\C^\ast)^n\to\C^\ast$, $\chi_\alpha(z)=z^\alpha$, has connected kernel. Hence $H_{\alpha,c}$ is a translate of a connected codimension-one algebraic subtorus and is irreducible.

The hypersurface is smooth. Indeed, for $f(z)=c_0+c_1z^\alpha$, one has
$
z_jf_{z_j}(z)=c_1\alpha_jz^\alpha.
$
Since $\alpha\neq0$ and $z^\alpha\neq0$, these logarithmic derivatives do not vanish simultaneously. On $H_{\alpha,c}$, the equality $c_1z^\alpha=-c_0$ gives
$
\gamma_H(z)
=[\alpha_1:\cdots:\alpha_n].
$
Thus the logarithmic Gauss map is constant, and its value belongs to $\R\PP^{n-1}$.

A constant map whose value lies in $\R\PP^{n-1}$ is not transverse to $\R\PP^{n-1}$. Indeed, $d\gamma_H=0$, so at every point $p\in H$,
$
d\gamma_H(T_pH)+T_{\gamma_H(p)}\R\PP^{n-1}
=T_{\gamma_H(p)}\R\PP^{n-1},
$
which is a proper subspace of $T_{\gamma_H(p)}\C\PP^{n-1}$. Therefore the transversality condition fails everywhere.

Nevertheless, every point of $H_{\alpha,c}$ is logarithmically critical because $\gamma_H(H)\subset\R\PP^{n-1}$. Thus $C_H=H$. Taking absolute values in $z^\alpha=-c_0/c_1$ and then logarithms gives
$
\langle\alpha,\Log z\rangle=\log|c_0/c_1|.
$
Conversely, every point $x\in\R^n$ satisfying this affine equation is realized by a point of $H_{\alpha,c}$: choose positive moduli $|z_j|=e^{x_j}$ and then choose phases whose character has argument $\arg(-c_0/c_1)$. Consequently,
$
\mathcal C\mathcal A_H=\mathcal A_H
=\{x\in\R^n:\langle\alpha,x\rangle=\log|c_0/c_1|\}.
$
This is a closed affine real-analytic hyperplane. Hence transversality of $\gamma_H$ to $\R\PP^{n-1}$ is not necessary.

We now consider the rank condition. Analyticity of an image does not require the parametrizing map to be an immersion. The elementary model is
$
\varphi:\R\to\R^2,\qquad \varphi(t)=(t^2,t^3).
$
At $t=0$, one has $d\varphi_0=0$, so the rank is zero. Nevertheless,
$
\varphi(\R)=\{(x,y)\in\R^2:y^2=x^3,\ x\geq0\}.
$
The inequality is actually redundant on the zero set $y^2=x^3$, because $x^3=y^2\geq0$ implies $x\geq0$. Therefore
$
\varphi(\R)=\{(x,y)\in\R^2:y^2-x^3=0\},
$
which is a real-analytic curve germ. Thus failure of immersion may produce an analytic cusp rather than a nonanalytic image.

This phenomenon occurs for amoeba contours of smooth algebraic curves. If $C\subset(\C^\ast)^2$ is smooth and its logarithmic critical locus $C_C$ is smooth, the points at which $\Log|_{C_C}$ is not an immersion form the logarithmic ramification set
$
\Sigma_C=\{p\in C_C:d(\Log|_{C_C})_p=0\}.
$
The images of such points are contour cusps; this is the local situation studied in \cite{Lang19,LangShapiroShustin21}. A complete ordinary cusp has a local real-analytic parametrization
$
t\longmapsto(t^pu(t),t^qv(t)),
$
where $u(0)v(0)\neq0$ and the primitive exponents produce a complete two-sided analytic image germ. At the preimage point, the differential vanishes, so the rank condition required by the immersed-critical-locus theorem fails. The image can nonetheless be the zero set of a convergent real-analytic equation and is therefore analytic. When all remaining contour branches are immersed and only finitely many complete cusp branches occur, the whole local contour is a finite union of analytic germs and hence is analytic. Thus maximal rank at every critical point is not necessary for analyticity of the contour.

The essential distinction is completeness. Compare the analytic cusp $t\mapsto(t^2,t^3)$ with the even fold $t\mapsto(t^2,0)$. The latter has image $\{(x,0):x\geq0\}$, which is not an analytic subset near the origin. If a real-analytic function vanishes on this half-axis, its restriction to the $x$-axis vanishes for $x\geq0$ and hence, by the one-variable identity theorem, vanishes for negative $x$ as well. Therefore no family of analytic equations can cut out only the half-axis. The cusp contains its complete analytic germ, whereas the fold contains only one half of its analytic continuation. Consequently, a rank drop is compatible with analyticity when the resulting image germ is complete, but it can destroy analyticity when it creates a one-sided branch.
\end{proof}

The preceding proof also shows that the two hypotheses play different roles. Transversality controls the structure of the source $C_H$. It ensures that $C_H=\gamma_H^{-1}(\R\PP^{n-1})$ is a smooth real-analytic manifold of dimension $n-1$. Analyticity of the contour, however, is a property of the image $\Log(C_H)$. A singular or higher-dimensional critical locus can still have a simple analytic image, as the binomial example demonstrates. Hence regularity of the source is not logically forced by regularity of the image.

The rank hypothesis controls the parametrization of each image branch. If $\rank d(\Log|_{C_H})=n-1$, the constant-rank theorem immediately makes the image branch a smooth embedded analytic hypersurface. If the rank drops, the constant-rank theorem no longer applies, but the image may still be analytic after eliminating the parameter. The cusp equation $y^2-x^3=0$ is the simplest illustration. Hence regularity of a parametrization is stronger than analyticity of its image.

\begin{proposition}  
For every smooth algebraic hypersurface $H\subset(\C^\ast)^n$, the contour $\mathcal C\mathcal A_H$ is closed, independently of both hypotheses.
\end{proposition}

\begin{proof}
The critical locus $C_H=\gamma_H^{-1}(\R\PP^{n-1})$ is closed in $H$ because $\R\PP^{n-1}$ is closed in $\C\PP^{n-1}$ and $\gamma_H$ is continuous. The map $\Log|_H:H\to\R^n$ is proper. Indeed, if $K\subset\R^n$ is compact, then there is an $M>0$ such that every point of $\Log^{-1}(K)$ satisfies $e^{-M}\leq|z_j|\leq e^M$. Thus $H\cap\Log^{-1}(K)$ is a closed subset of a compact product of annuli and is compact. The restriction $\Log|_{C_H}$ is therefore proper. A proper continuous map between locally compact Hausdorff spaces is closed, so $\Log(C_H)$ is closed.
\end{proof}

Thus neither transversality nor maximal rank is needed for the word ``closed.'' Their purpose in the immersed-critical-locus theorem is to prove local analyticity by producing finitely many smooth immersed branches.

\begin{criterion}%
The contour $\mathcal C\mathcal A_H$ is a real-analytic hypersurface near $x\in\mathcal C\mathcal A_H$ if the logarithmic critical locus above a neighborhood of $x$ has finitely many relevant real-analytic branches $B_1,\ldots,B_r$, each image germ $\Log(B_j)$ is a complete real-analytic hypersurface germ, and these germs account for every contour point near $x$. Immersed branches satisfy this condition automatically, but complete ramified branches may satisfy it as well.
\end{criterion}

\begin{proof}
Suppose the $j$-th image germ is defined on a common neighborhood $U$ of $x$ by real-analytic equations $g_{j1}=\cdots=g_{jm_j}=0$. The finite union of these image germs is analytic. Indeed, it is the common zero set of all products
$
g_{1\ell_1}\cdots g_{r\ell_r},
$
where $1\leq\ell_j\leq m_j$. If a point belongs to one branch, every such product vanishes. Conversely, if it belongs to none, one can select for every $j$ a nonvanishing $g_{j\ell_j}$, and the corresponding product does not vanish. The branch-completeness assumption identifies this finite union with the full contour near $x$.
\end{proof}

There is also a useful partial converse. Suppose the contour is not only an analytic hypersurface but a smooth embedded real-analytic hypersurface near $x$, and suppose $p\in C_H$ is the unique point above $x$ and $C_H$ is a smooth $(n-1)$-manifold near $p$. If $\Log|_{C_H}$ parametrizes the contour locally with local topological degree one, then one expects rank $n-1$ at $p$; otherwise the inverse parametrization cannot be a local analytic diffeomorphism. This conclusion requires the uniqueness and local-degree hypotheses. Without them, a ramified parametrization such as $t\mapsto t^3$ can cover a smooth analytic line while having zero derivative at the origin. Therefore even smoothness of the image alone does not force the rank condition for an arbitrary parametrizing branch.

{\bf Conclusion.}
The two conditions
$
\gamma_H\pitchfork\R\PP^{n-1}
$
and
$
\rank d(\Log|_{C_H})_p=n-1 $
for every  $p\in C_H$
are sufficient but not necessary for $\mathcal C\mathcal A_H$ to be a closed real-analytic hypersurface. Transversality may be replaced by any hypothesis yielding finitely many manageable analytic branches of the critical locus. The rank condition may be replaced by the weaker requirement that every ramified image branch be a complete analytic germ. The true obstruction is not failure of transversality or rank by itself, but the appearance of singular, one-sided, or incomplete image germs that cannot be defined using analytic equalities alone.


\section{Necessary Conditions for Analytic Amoeba Contours} 

Let $H\subset(\C^\ast)^n$ be an algebraic hypersurface, let $H_{\mathrm{sm}}$ be its smooth locus, and put
$
C_H=\Crit(\Log|_{H_{\mathrm{sm}}})
$
and
$
\mathcal C\mathcal A_H=\Log(C_H).
$
There is no known single differential condition on the logarithmic Gauss map that is both necessary and sufficient for $\mathcal C\mathcal A_H$ to be a closed real-analytic hypersurface. In particular, transversality of $\gamma_H$ to $\R\PP^{n-1}$ and maximal rank of $\Log|_{C_H}$ are not necessary. Nevertheless, analyticity imposes several rigorous necessary conditions on the contour germ and, indirectly, on the logarithmic critical locus. These conditions are weaker than transversality and immersion because they concern the image rather than every parametrizing branch.

Throughout all this work, a real-analytic hypersurface means a closed real-analytic subset that is locally of pure real dimension $n-1$ and locally is the zero set of a nonzero real-analytic function. Singularities and finitely many local branches are allowed.

\medskip

In the theorem below, the tangent cone is the \emph{algebraic tangent cone} of the real-analytic germ, namely the real affine cone defined by the initial homogeneous ideal of its local analytic ideal. Indeed, the real-analytic cusp $\{(x,y)\in\R^2:y^2=x^3\}$ has the positive $x$-axis as its set-theoretic tangent cone, and that half-axis is not a real algebraic set. Its algebraic tangent cone is the full $x$-axis, defined by $y^2=0$. Thus the distinction is essential.%

\begin{theorem}%
Let $H\subset(\C^\ast)^n$ be a smooth algebraic hypersurface, let
$
C_H=\Crit(\Log|_H),
$
and let $\mathcal C\mathcal A_H=\Log(C_H)$. Assume that $\mathcal C\mathcal A_H$ is a closed real-analytic hypersurface of $\R^n$, where ``hypersurface'' means a reduced real-analytic set of pure local dimension $n-1$. Then the following hold:
\begin{itemize}
\item[(i)]\, Every contour germ has pure local dimension $n-1$ and finitely many local real-analytic irreducible branches. Its regular locus is open and dense in every local component, while its singular locus has local dimension at most $n-2$. 

\item[(ii)] \, At every regular contour point, the contour is a smooth embedded real-analytic hypersurface without boundary and locally separates the ambient space into two sides.

\item[(iii)]\,  The algebraic tangent cone at every contour point is a real algebraic cone of dimension $n-1$; its nonreduced scheme structure or a chosen analytic stratification may contain lower-dimensional subsidiary pieces. No contour branch may occur only as one side of a larger analytic continuation. 

\item[(iv)]\, Finally, on a dense open subset of every $(n-1)$-dimensional contour component, at least one real-analytic stratum of $C_H$ maps by $\Log$ with rank $n-1$.

\end{itemize}
\end{theorem}

\begin{proof}
Put $X=\mathcal C\mathcal A_H$. The hypothesis that $X$ is a real-analytic subset of $\R^n$ means that, for every $x\in X$, there are an open neighborhood $U\subset\R^n$ of $x$ and finitely many real-analytic functions $g_1,\ldots,g_s\in\mathcal O_{\R^n}(U)$ such that
$$
X\cap U=\{y\in U:g_1(y)=\cdots=g_s(y)=0\}.
$$
The local analytic ideal of the germ $X_x$ is
$$
\mathcal I_{X,x}=\{g\in\mathcal O_{\R^n,x}:g|_{X_x}=0\}.
$$
Because $X$ is assumed reduced, this is the full vanishing ideal of the germ. Because $X$ is assumed to be a real-analytic hypersurface in the pure-dimensional sense, every sufficiently small local component of $X_x$ has real-analytic dimension $n-1$. This already proves the purity assertion, but the remaining consequences require a closer examination of the local analytic structure.

The local ring $\mathcal O_{\R^n,x}$ of real-analytic germs at $x$ is Noetherian. One way to see this is to identify it, after translating $x$ to the origin, with the ring $\R\{u_1,\ldots,u_n\}$ of convergent real power series and then apply the Weierstrass preparation and division theorems. Hence the quotient
$$
\mathcal O_{X,x}=\mathcal O_{\R^n,x}/\mathcal I_{X,x}
$$
is also Noetherian. A radical ideal in a Noetherian ring has only finitely many minimal prime ideals. If $\mathfrak p_1,\ldots,\mathfrak p_r$ are the minimal prime ideals over $\mathcal I_{X,x}$, the germs determined by these primes are precisely the local irreducible analytic branches of the reduced germ $X_x$. It follows that $X_x$ has only finitely many local irreducible branches. The purity hypothesis says that every one of these branches has local dimension $n-1$.

We next examine the regular and singular loci. A point $y\in X$ is regular if there is a neighborhood $V$ of $y$ such that $X\cap V$ is a real-analytic submanifold of $\R^n$ of dimension $n-1$. Equivalently, after choosing local analytic generators $h_1,\ldots,h_q$ of the ideal of $X$, the differentials $dh_j(y)$ have the locally maximal rank needed to cut out a manifold of dimension $n-1$. The relevant nonvanishing condition is expressed by a suitable Jacobian minor. Since nonvanishing of an analytic function is an open condition, $\Reg(X)$ is open in $X$.

Let $B$ be any local irreducible branch of $X_x$. Its singular locus is a proper real-analytic subset of $B$. To justify the word ``proper,'' suppose conversely that every point of a nonempty relatively open subset of $B$ were singular. Take a local analytic presentation of the reduced branch and choose generators of its prime ideal. The vanishing of all Jacobian minors of the expected rank on a relatively open subset of the irreducible branch forces those minors, by the real-analytic identity principle on the regular strata and analytic continuation, to vanish identically on the branch. The Jacobian criterion, applied after complexifying the convergent power-series ideal and then returning to the real germ, would force the local dimension of the reduced branch to be larger than its assumed dimension $n-1$, a contradiction. Equivalently, one may invoke the regular-stratification theorem for real-analytic sets: every reduced real-analytic set admits a locally finite real-analytic stratification, and the union of the strata of maximal dimension is open and dense in each pure-dimensional local component. Therefore $\Reg(X)\cap B$ is dense in $B$.

Since $B$ has dimension $n-1$ and $B\setminus\Reg(X)$ is a proper analytic subset of $B$, the dimension-drop theorem for real-analytic sets gives
$$
\operatorname{dim}_{\mathrm{loc},y}\bigl(B\setminus\Reg(X)\bigr)\leq n-2
$$
at every point $y$ of that singular subset. There are only finitely many local branches. Their singular loci have dimension at most $n-2$, and the intersection of two distinct branches also has dimension at most $n-2$; otherwise two irreducible branches of dimension $n-1$ would agree on a relatively open subset and hence would define the same germ. The singular locus of the union is contained in the union of the singular loci of the branches and their pairwise intersections. Consequently,
$$
\operatorname{dim}_{\mathrm{loc},y}\Sing(X)\leq n-2
$$
for every $y\in\Sing(X)$.

Let $x\in\Reg(X)$. By definition, $X$ is a real-analytic embedded submanifold of codimension one near $x$. There is therefore a real-analytic function $\rho$ on a neighborhood $U$ of $x$ for which $d\rho(x)\neq0$ and
$$
X\cap U=\{y\in U:\rho(y)=0\}.
$$
After shrinking $U$, continuity gives $d\rho(y)\neq0$ for every $y\in U$. The real-analytic implicit-function theorem supplies real-analytic coordinates $(u_1,\ldots,u_n)$ centered at $x$ in which $\rho=u_n$. In these coordinates, $X\cap U$ is $\{u_n=0\}$. It has no boundary as a manifold, because a neighborhood of each of its points is analytically diffeomorphic to an open subset of $\R^{n-1}$, rather than to a half-space. Moreover,
$$
U\setminus X=\{u_n>0\}\,\cup\,\{u_n<0\}.
$$
After taking $U$ to be a sufficiently small coordinate ball, both sets on the right are connected and nonempty. Thus the contour locally separates $\R^n$ into exactly two sides at every regular point.

We now prove the tangent-cone assertion under the convention stated above. Translate $x$ to the origin and write $\mathfrak m$ for the maximal ideal of $\mathcal O_{\R^n,x}$. If $0\neq g\in\mathcal I_{X,x}$ has convergent expansion $g=g_\nu+g_{\nu+1}+\cdots$, where every $g_j$ is homogeneous of degree $j$ and $g_\nu\neq0$, let $\operatorname{in}(g)=g_\nu$. The initial ideal is the homogeneous ideal
$$
\operatorname{in}(\mathcal I_{X,x})
=\bigl(\operatorname{in}(g):g\in\mathcal I_{X,x}\bigr)
\subset\R[u_1,\ldots,u_n].
$$
The algebraic tangent cone is
$$
C_x^{\mathrm{alg}}X
=V_{\R}\bigl(\operatorname{in}(\mathcal I_{X,x})\bigr).
$$
It is a real algebraic cone because its defining ideal is homogeneous. The associated graded ring of the local analytic ring is naturally
$$
\operatorname{gr}_{\mathfrak m}\mathcal O_{X,x}
\simeq
\R[u_1,\ldots,u_n]\big/\operatorname{in}(\mathcal I_{X,x}).
$$
Passing from a Noetherian local ring to its associated graded ring preserves Krull dimension. Hence
$$
\dim \operatorname{gr}_{\mathfrak m}\mathcal O_{X,x}
=\dim\mathcal O_{X,x}
=n-1.
$$
Thus the algebraic tangent cone has algebraic dimension $n-1$. The initial ideal need not be radical, even when the original analytic germ is reduced, and the cone can therefore carry multiplicities or embedded lower-dimensional associated components. A real-analytic or Whitney stratification of its real support may likewise contain strata of dimension smaller than $n-1$. These are the lower-dimensional subsidiary pieces mentioned in the statement; they do not alter the dimension of the cone.

We next formalize the assertion that an analytic branch cannot stop after occupying only one side of an analytic continuation. Let $B$ be an irreducible branch of $X_x$, and suppose that there is a connected smooth real-analytic hypersurface germ $M_x$ such that $B$ contains a nonempty relatively open subset $W$ of $M$ whose closure in $M$ contains $x$. Every germ $g\in\mathcal I_{B,x}$ vanishes on $W$. The restriction $g|_M$ is real analytic. By the identity theorem on the connected real-analytic manifold $M$, it vanishes on the entire germ $M_x$. Therefore $\mathcal I_{B,x}\subset\mathcal I_{M,x}$, which gives $M_x\subset B_x$. Both germs have dimension $n-1$. Since $B_x$ is irreducible and $M_x$ already contains a relatively open part of it, equality follows: $B_x=M_x$. In particular, a putative half-hypersurface with an analytic continuation across its edge cannot itself be a local real-analytic branch. This is the precise analytic meaning of the no-one-sided-branch assertion.

It remains to prove the rank statement. The map $\Log:(\C^\ast)^n\to\R^n$ is real analytic. It is also proper. Indeed, if $K\subset\R^n$ is compact, then every $z=(z_1,\ldots,z_n)\in\Log^{-1}(K)$ has each modulus $|z_j|$ bounded above and bounded away from zero. Writing $z_j=e^{x_j+i\theta_j}$ identifies $\Log^{-1}(K)$ with $K\times(S^1)^n$, and this set is compact. Since $H$ is closed in $(\C^\ast)^n$ and $C_H$ is the zero set in $H$ of the appropriate minors of $d(\Log|_H)$, the set $C_H$ is closed. Consequently, $\Log|_{C_H}:C_H\to\R^n$ is proper.

The critical locus $C_H$ is a real-analytic, hence semianalytic, subset of the real-analytic manifold underlying $H$. Choose a locally finite real-analytic stratification $\mathscr S$ of $C_H$ such that the rank of $\Log|_S$ is constant on every stratum $S\in\mathscr S$. Such a stratification is obtained by first stratifying the analytic rank loci defined by the minors of $d\Log$ and then refining to a Whitney stratification. Because every point of $C_H$ is critical for $\Log|_H$ and $\dim_{\R}H=2n-2$, one has
$$
\rank d(\Log|_S)\leq n-1
$$
on every stratum.

Fix an irreducible local component $B$ of $X$ of dimension $n-1$ and choose a relatively compact open neighborhood $U\subset\R^n$ meeting only the finitely many contour branches under consideration. Properness implies that $\Log^{-1}(\overline U)\cap C_H$ is compact. Local finiteness of $\mathscr S$ therefore implies that only finitely many strata meet this compact set. Let $E_U$ be the union, over those strata $S$ for which $\rank d(\Log|_S)\leq n-2$, of $\Log(S)\cap U$. By the constant-rank theorem, the local dimension of $\Log(S)$ is at most $n-2$. Since only finitely many strata occur over $\overline U$, the closure $\overline{E_U}$ is a subanalytic subset of $U$ of dimension at most $n-2$.

Remove from the regular part of $B\cap U$ the set $\overline{E_U}$, the intersections with the other local components of $X$, and the lower-dimensional frontiers of the maximal strata in a stratification of $X$. The resulting set $U_B^\circ$ is relatively open in $B$, and it is dense because every removed set has local dimension at most $n-2$, whereas $B$ has dimension $n-1$. Take $y\in U_B^\circ$. Since $y\in X=\Log(C_H)$, there exists $p\in C_H$ with $\Log(p)=y$. Let $S$ be the stratum containing $p$. If $\rank d(\Log|_S)_p\leq n-2$, then $y\in E_U\subset\overline{E_U}$, contrary to the definition of $U_B^\circ$. Hence
$$
\rank d(\Log|_S)_p=n-1.
$$
Thus every point of the dense open set $U_B^\circ$ has at least one preimage lying on a real-analytic stratum of $C_H$ on which $\Log$ has rank $n-1$. Applying this construction in a locally finite family of neighborhoods and taking the union of the resulting open subsets proves the assertion on a dense open subset of every $(n-1)$-dimensional contour component.
All the conclusions now follow.
\end{proof}


\begin{proposition}%
If $H\subset(\C^\ast)^n$ is smooth, then $\mathcal C\mathcal A_H$ is closed independently of whether it is analytic. Thus, in the smooth case, closedness is automatic and is not an additional necessary restriction on the logarithmic Gauss map.
\end{proposition}

\begin{proof}
For a smooth hypersurface,
$
C_H=\gamma_H^{-1}(\R\PP^{n-1}).
$
The real projective space is closed in $\C\PP^{n-1}$, so $C_H$ is closed in $H$. The restriction $\Log|_H:H\to\R^n$ is proper. Indeed, if $K\subset\R^n$ is compact, then there is an $M>0$ such that every $z\in\Log^{-1}(K)$ satisfies
$
e^{-M}\leq|z_j|\leq e^M
$
for all $j$. Hence $H\cap\Log^{-1}(K)$ is a closed subset of a compact product of annuli and is compact. The restriction $\Log|_{C_H}$ is therefore proper. A proper continuous map into a Hausdorff space has closed image, and consequently
$
\mathcal C\mathcal A_H=\Log(C_H)
$
is closed.
\end{proof}

For a singular hypersurface, closedness is no longer automatic when the contour is defined using only $H_{\mathrm{sm}}$. Smooth critical points may converge to a singular point that has been removed from the domain, and the limiting logarithmic value may have no other smooth preimage. Therefore the following condition is necessary.

\begin{proposition}%
Suppose $H$ is singular and $\mathcal C\mathcal A_H=\Log(\Crit(\Log|_{H_{\mathrm{sm}}}))$ is closed. If $p_\nu\in C_H$ converges in $H$ to a singular point $p\in\Sing(H)$ and $\Log(p_\nu)\to x$, then $x$ must be the logarithmic image of some smooth critical point. Equivalently,
$
\overline{\Log(C_H)}=\Log(C_H).
$
\end{proposition}

\begin{proof}
Every $\Log(p_\nu)$ belongs to the contour. Closedness forces its limit $x$ to belong to the contour. By definition of the contour, there must exist $q\in C_H\subset H_{\mathrm{sm}}$ such that $\Log(q)=x$. If no such $q$ existed, $x$ would be a missing limit point and the contour would not be closed.
\end{proof}

The preceding condition is necessary only in the singular case. It does not require the singular point $p$ itself to be included in the critical locus; it requires its limiting logarithmic value to be recovered by some smooth critical point. If one defines instead a stratified contour using all strata of a Whitney stratification, the formulation changes, but local closedness remains necessary for analyticity.
There is a further necessary compatibility condition at branch intersections. If finitely many contour branches meet at $x$, their union must be the zero set of an analytic ideal. For hypersurface branches individually defined by $g_j=0$, this condition is satisfied when the entire local union is defined by $g_1\cdots g_r=0$. What is forbidden is selecting only inequality-defined portions of these germs. Thus transverse crossings, tangential crossings, and complete cusps may be analytic, whereas a crossing from which one half-branch has been deleted is not.

\begin{proposition}%
Let $Y$ be an irreducible local component of $\mathcal C\mathcal A_H$ of dimension $n-1$. If $Y$ is analytic, then there is a dense open subset $Y^\circ\subset\Reg(Y)$ such that for every $x\in Y^\circ$ at least one point $p\in C_H\cap\Log^{-1}(x)$ lies on a real-analytic stratum $S\subset C_H$ satisfying $\rank d(\Log|_S)_p=n-1$.
\end{proposition}

\begin{proof}
We first choose a precise representative of the local component. Let $x_0$ be the point at which the germ $Y$ is considered. After replacing a neighborhood of $x_0$ by a smaller relatively compact open neighborhood $U\subset\R^n$, the germ $Y$ is represented by a closed irreducible real-analytic subset, still denoted by $Y$, of $U$, and $\dim Y=n-1$. We may also arrange that only finitely many local components of $\mathcal C\mathcal A_H\cap U$ meet a smaller neighborhood of $x_0$. Since $Y$ is an irreducible analytic set, its regular locus $\Reg(Y)$ is open and dense in $Y$, and its singular locus has dimension at most $n-2$. These are standard consequences of the locally finite real-analytic stratification theorem \cite{Hardt75,Hironaka73}.

The essential point is to show that the image of the part of $C_H$ on which $\Log$ has rank at most $n-2$ cannot contain an open subset of $Y$. This requires both a rank-adapted stratification and local properness of the logarithmic map.

The map $\Log:(\C^\ast)^n\to\R^n$ is proper. Indeed, let $K\subset\R^n$ be compact. The map that sends $z=(z_1,\ldots,z_n)$ to the pair formed by $\Log(z)$ and $(z_1/|z_1|,\ldots,z_n/|z_n|)$ identifies $\Log^{-1}(K)$ homeomorphically with $K\times(S^1)^n$. Thus $\Log^{-1}(K)$ is compact.

Because $H$ is smooth, the differential $d(\Log|_H)$ is a real-analytic bundle morphism from $TH$ to the trivial bundle $H\times\R^n$. Its critical locus $C_H$ is locally defined by the vanishing of the maximal minors of a real-analytic matrix representing this morphism. Hence $C_H$ is a closed real-analytic subset of $H$. It follows that $\Log|_{C_H}:C_H\to\R^n$ is proper.

Choose a locally finite real-analytic Whitney stratification $\mathscr S$ of $C_H$ that is adapted to all rank loci of $\Log|_{C_H}$. Concretely, for each integer $r$, the condition $\rank d(\Log|_H)\leq r$ is given locally by the vanishing of the $(r+1)\times(r+1)$ minors of a real-analytic matrix, and one first chooses a stratification compatible with these analytic rank loci. Refining it if necessary gives a Whitney stratification on which $\rank d(\Log|_S)$ is constant for every stratum $S\in\mathscr S$. The existence of such a compatible stratification is a standard form of the real-analytic mapping-stratification theorem \cite{Hardt75}; the subanalytic properties of its images under proper real-analytic maps are established in \cite{Hironaka73,BM88}.

For every $p\in C_H$, the differential $d(\Log|_H)_p$ has rank at most $n-1$, because $p$ is a critical point of a map with target dimension $n$. Since $T_pS\subset T_pH$ for the stratum $S$ containing $p$, restriction cannot increase rank. Consequently, $\rank d(\Log|_S)_p\leq n-1$ on every stratum.

Choose an open set $V$ with $x_0\in V$ and compact closure $\overline V\subset U$. Properness gives that
$$
K_V=C_H\cap\Log^{-1}(\overline V)
$$
is compact. Since $\mathscr S$ is locally finite, only finitely many strata of $\mathscr S$ meet $K_V$. Denote these strata by $S_1,\ldots,S_N$, and let $I_{\mathrm{low}}$ be the set of indices $j$ for which the constant rank of $\Log|_{S_j}$ is at most $n-2$. Define
$$
E_V=\bigcup_{j\in I_{\mathrm{low}}}\Log(S_j\cap\Log^{-1}(V)).
$$

We claim that $\dim E_V\leq n-2$. Fix $j\in I_{\mathrm{low}}$ and write $r_j=\rank d(\Log|_{S_j})\leq n-2$. The real-analytic constant-rank theorem shows that, in suitable local analytic coordinates on $S_j$ and on $\R^n$, the restriction $\Log|_{S_j}$ has the local form
$$
(t_1,\ldots,t_{\dim S_j})\longmapsto(t_1,\ldots,t_{r_j},0,\ldots,0).
$$
Its local image therefore has dimension $r_j\leq n-2$. The set $\Log(S_j\cap\Log^{-1}(V))$ is subanalytic, because the map is real analytic and the relevant restriction is relatively proper over compact subsets of $V$. A finite union of subanalytic sets is subanalytic, and the dimension of a finite union is the maximum of the dimensions of its members. Hence $E_V$ is subanalytic and $\dim E_V\leq n-2$. The closure theorem for subanalytic sets gives $\dim\overline{E_V}=\dim E_V\leq n-2$.

Set
$$
Z_V=Y\cap\overline{E_V}.
$$
Since $\dim Y=n-1$ and $\dim\overline{E_V}\leq n-2$, the set $Z_V$ cannot contain a nonempty relatively open subset of $Y$. In particular, $Y\setminus Z_V$ is dense in $Y$. The set
$$
Y_V^\circ=\Reg(Y)\cap V\setminus\overline{E_V}
$$
is relatively open in $\Reg(Y)\cap V$ and dense in $Y\cap V$, because both $Y\setminus\Reg(Y)$ and $Z_V$ have dimension at most $n-2$.

Let $x\in Y_V^\circ$. Since $Y\subset\mathcal C\mathcal A_H$ as a local contour component, $x\in\Log(C_H)$. Therefore there exists $p\in C_H$ such that $\Log(p)=x$. Because $x\in V$, the point $p$ belongs to $K_V$ and hence lies on one of the finitely many strata $S_1,\ldots,S_N$. Let $S$ be the stratum containing $p$. If $\rank d(\Log|_S)_p\leq n-2$, then $S=S_j$ for some $j\in I_{\mathrm{low}}$, and the equality $\Log(p)=x$ would imply $x\in E_V\subset\overline{E_V}$. This contradicts the choice of $x$. We must therefore have $\rank d(\Log|_S)_p\geq n-1$. The opposite inequality was established above, so
$
\rank d(\Log|_S)_p=n-1.
$

This proves the assertion on the dense relatively open subset $Y_V^\circ$ of the chosen representative near $x_0$. To formulate the conclusion intrinsically on the whole local component, cover $\Reg(Y)$ by a countable locally finite family of relatively compact neighborhoods $V_\alpha\Subset U$ of the preceding type, and define $Y^\circ$ as the union of the corresponding sets $Y_{V_\alpha}^\circ$. This union is open in $\Reg(Y)$. It is dense in $Y$ because, in every $V_\alpha$ meeting $Y$, its complement is contained in a subanalytic set of dimension at most $n-2$. For every $x\in Y^\circ$, the construction supplies a point $p\in C_H\cap\Log^{-1}(x)$ and a real-analytic stratum $S\subset C_H$ through $p$ such that $\rank d(\Log|_S)_p=n-1$. The proposition follows.
\end{proof}

{\bf Summary of the necessary conditions.}
If $\mathcal C\mathcal A_H$ is a closed real-analytic hypersurface of $\R^n$, then it must be locally closed, pure of dimension $n-1$, locally finite in its analytic branch structure, and free of uncompleted one-sided image germs. Its regular locus must be dense, its reduced singular locus must have codimension at least one inside the contour, and every tangent cone must be real algebraic. On each contour component, the logarithmic map must attain rank $n-1$ on some critical stratum over a dense open subset, although rank may drop at exceptional analytic cusps or other complete ramified germs. If $H$ is singular, no limiting critical value may be lost at $\Sing(H)$. These conditions are necessary but not, individually or collectively in this elementary form, sufficient; sufficiency additionally requires that the local image branches glue as complete analytic germs.


\begin{thebibliography}{99}

\bibitem{BM88}
E. Bierstone and P.~D. Milman,
Semianalytic and subanalytic sets,
\emph{Inst. Hautes \'Etudes Sci. Publ. Math.} \textbf{67} (1988), 5--42.

 
\bibitem{BCR98}
J. Bochnak, M. Coste, and M.-F. Roy,
\emph{Real Algebraic Geometry},
Ergebnisse der Mathematik und ihrer Grenzgebiete, vol.~36,
Springer-Verlag, Berlin, 1998.

\bibitem{Hardt75}
R.~M. Hardt,
Stratification of real analytic mappings and images,
\emph{Invent. Math.} \textbf{28} (1975), 193--208.

\bibitem{Hironaka73}
H. Hironaka,
Subanalytic sets,
in \emph{Number Theory, Algebraic Geometry and Commutative Algebra},
in honor of Yasuo Akizuki,
Kinokuniya, Tokyo, 1973, pp.~453--493.


\bibitem{Lang19}
L. Lang,
Amoebas of curves and the Lyashko--Looijenga map,
\emph{J. Lond. Math. Soc.} (2) \textbf{100} (2019), no.~1, 301--322.

\bibitem{LangShapiroShustin21}
L. Lang, B. Shapiro, and E. Shustin,
On the number of intersection points of the contour of an amoeba with a line,
\emph{Indiana Univ. Math. J.} \textbf{70} (2021), no.~4, 1335--1353.

\bibitem{Mikhalkin00}
G. Mikhalkin,
Real algebraic curves, the moment map and amoebas,
\emph{Ann. of Math.} (2) \textbf{151} (2000), no.~1, 309--326.

\bibitem{Mikhalkin04}
G. Mikhalkin,
Amoebas of algebraic varieties and tropical geometry,
in \emph{Different Faces of Geometry},
International Mathematical Series, vol.~3,
Kluwer Academic/Plenum Publishers, New York, 2004, pp.~257--300.


\bibitem{Mikhalkin04Pants}
G. Mikhalkin,
Decomposition into pairs-of-pants for complex algebraic hypersurfaces,
\emph{Topology} \textbf{43} (2004), no.~5, 1035--1065.
\href{https://doi.org/10.1016/j.top.2003.11.006}
{https://doi.org/10.1016/j.top.2003.11.006}.



\bibitem{Whitney65}
H. Whitney,
Tangents to an analytic variety,
\emph{Ann. of Math.} (2) \textbf{81} (1965), 496--549.

\end{thebibliography}
\end{document}